\documentclass[12pt]{amsart}
\usepackage{hyperref}
\usepackage{amsmath}
\usepackage{amsthm}
\usepackage{amsfonts}
\usepackage{amssymb}
\usepackage{graphicx}
\DeclareGraphicsExtensions{.eps}

\newtheorem{theorem}{Theorem}[section]
\newtheorem{hyp*}{Conjecture}
\newtheorem{lemma}[theorem]{Lemma}
\newtheorem{example}[theorem]{Example}
\newtheorem{prop}[theorem]{Proposition}
\newtheorem{corollary}[theorem]{Corollary}

\def\sli{\sum\limits}

\def\g{\gamma}

\def\la{\lambda}

\def\a{\alpha}
\def\R{\mathbb{R}}

\def\vf{\varphi}

\def\s{\sigma}

\def\cC{\mathcal{C}}
\def\cD{\mathcal{D}}

\def\J{\mathcal{J}}

\newcommand{\ci}[1]{_{ {}_{\scriptstyle #1}}}

\newcommand{\supp}{\operatorname{supp}}

\newcommand{\al}{\alpha}

\newcommand{\diam}{\operatorname{diam}}

\newcommand{\C}{\mathbb{C}}

\newcommand{\e}{\varepsilon}
\newcommand{\dist}{\operatorname{dist}}

\newcommand{\HH}{\mathcal{H}}

\newcommand{\wt}{\widetilde}

\newcommand{\cE}{\mathcal{E}}

\def\cyr{\fontencoding{OT2}\fontfamily{wncyr}\selectfont}
\DeclareTextFontCommand{\textcyr}{\cyr}

{\end{list}}

\renewcommand{\Im}{\textup{Im}}

\newcounter{vremennyj}

\numberwithin{equation}{section}

\begin{document}

\title[Cauchy independent measures]{Cauchy independent measures and almost-additivity of analytic capacity}

\author{Vladimir Eiderman}
\address{Vladimir Eiderman, Indiana University, Bloomington, Indiana, USA}
\email{veiderma@indiana.edu}

\author{Alexander Reznikov}
\address{Alexander Reznikov, Michigan State University, East Lansing, Michigan, USA}
\email{rezniko2@msu.edu}

\author{Alexander Volberg}
\address{Alexander Volberg, Michigan State University, East Lansing, Michigan, USA}
\email{volberg@math.msu.edu}

\thanks{ Alexander Volberg \ was partially supported by the U.S.\ NSF grant DMS-0758552}

\begin{abstract}
We show that, given a family of discs centered at a chord-arc curve, the analytic capacity of a union of arbitrary subsets of these discs (one subset in each disc) is comparable with the sum of their analytic capacities. However we need that the discs in the question would be separated, and it is not clear whether the separation condition is essential or not. We apply this result to find families $\{\mu_j\}$ of measures in $\C$ with the following property. If the Cauchy integral operators $\mathcal{C}_{\mu_j}$ from $L^2(\mu_j)$ to itself are bounded uniformly in $j$, then $\mathcal{C}_\mu$, $\mu=\sum\mu_j$, is also bounded from $L^2(\mu)$ to itself.
\end{abstract}

\maketitle

\section{Introduction}\label{int}

We consider two properties of families of sets and measures in the complex plane.

\subsection{Almost additivity of analytic capacity}\label{one}

The {\it analytic capacity} $\gamma(F)$ of a compact set $F$ in $\C$ is defined by the equality
$$
\gamma(F) = \sup|f'(\infty)|,
$$
where the supremum is taken over all analytic functions $f\colon \C\setminus F \to \C$ with $|f|\le 1$ on $\C\setminus F$. Here $f'(\infty)=\lim_{z\to \infty}z(f(z)-f(\infty))$. For non-compact $F$ we set
$$
\gamma(F) = \sup\{\gamma(K): K\text{ compact, }K\subset F\}
$$
\cite{G}. For a summary of equivalent definitions the reader can see \cite{Tolsa-book} and \cite{Volberg}.

In the celebrated paper \cite{Tolsa-sem} Tolsa established the countable semiadditivity of the analytic capacity, i.~e. that
$$
\g\big(\bigcup F_i\big)\le C\sum\g(F_i)
$$
with an absolute constant $C$. But the inverse inequality does not hold in general. To see that we consider the $n$-th generation $E_n^{1/4}$ of the corner 1/4-Cantor set constructed in the following way. Start with the unit square (0-th generation). The $j$-th generation consists of $4^j$ squares $E_{j,k}$ with side length $4^{-j}$, each square $E_{j,k}$ contains four squares of $(j+1)$-th generation, located at the corners of $E_{j,k}$, and so on. It's known \cite{MTV} that $\g(\bigcup_{k=1}^{4^n}E_{n,k})=\g(E_n^{1/4})\asymp 1/\sqrt n$ with absolute constants of comparison; here $P\asymp Q$ means that $cP\le Q\le CP$. Positive constants $c$, $C$, $a$, $A$ (possibly with indices) are not necessarily the same at each appearance. On the other hand,
$$
\sum_{k=1}^{4^n}\g(E_{n,k})\asymp4^n\cdot 4^{-n}=1.
$$
Thus, ``almost additivity'' $\g\big(\bigcup F_i\big)\asymp\sum\g(F_i)$ of the analytic capacity does not take place in general. N.~A.~Shirokov posed the question on the validity of this property for the special class of sets described in the following Theorem \ref{superth}.

We say that $\Gamma$ is a chord-arc curve, if
\begin{equation}\label{chordarc}
|t-s|\le A_0|z(t)-z(s)|,\quad A_0>1,
\end{equation}
where $z(t)$ is the arc-length parametrization of $\Gamma$.

\begin{theorem}\label{superth}
Let $D_j$ be discs with centers on a chord-arc curve $\Gamma$, such that $\la D_j \cap \la D_k = \emptyset$, $j\not=k$, for some $\la>1$. Let $E_j\subset D_j$ be arbitrary compact sets. Then there exists a constant $c=c(\la, A_0)$, such that
\begin{equation}\label{almadd0}
\gamma\big(\bigcup E_j\big) \ge c \sum_j \gamma(E_j).
\end{equation}
\end{theorem}

We conclude the present subsection with a useful corollary.

\begin{corollary}\label{circle}
If $\Gamma$ in Theorem \ref{superth} is the union of $n$ chord-arc curves $\Gamma_i$ with the same constant $A_0$, then
$$
\gamma\big(\bigcup E_j\big) \ge \frac cn \sli_j \gamma(E_j),\quad c=c(\la, A_0).
$$
In particular, $\Gamma$ might be a circle ($n=2$). Moreover, $\{D_j\}$ might be a family of $\lambda$-separated discs (that is $\la D_j$ are disjoint), $\la>1$, intersecting the same circle $\Gamma$, not necessarily with centers on $\Gamma$.
\end{corollary}

\begin{proof} Let $D_j=D(x_j,r_j)$, and let $m$ be such that
$$
\max_{1\le k\le n}\big\{\sum_{j:x_j\in\Gamma_k}\gamma(E_j)\big\}=\sum_{j:x_j\in\Gamma_m} \gamma(E_j).
$$
Then by Theorem \ref{superth} we have
$$
\gamma\big(\bigcup E_j\big) \ge\gamma\big(\bigcup_{j:x_j\in\Gamma_m}  E_j\big) \ge c \sli_{j:x_j\in\Gamma_m} \gamma(E_j) \ge  \frac cn \sli_j \gamma(E_j).
$$

For the last assertion it is sufficient to prove that if all discs  $D_j$ intersect a semicircle $T$, then there is another chord-arc curve $\widetilde\Gamma$ with a constant $A_0=A_0(\lambda)$ containing all centers of $D_j$. We may assume that there are more than one disc $D_j$. Fix $j$, and let $a_j,b_j$ be the points of intersection of $T$ and the circle $\partial(\la'D_j)$, where $\lambda'=\frac{1+\lambda}2$ (for two discs $D_j$ it might be only one such a point). We replace the arc of $T$ with the endpoints $a_j,b_j$ by two line segments $x_ja_j$ and $x_jb_j$ (and by the only one line segment in the case of one point of intersection). We claim that the obtained curve is chord-arc. Indeed, since $|a_j-b_j|\ge c(\la)r_j$, the condition \eqref{chordarc} holds with $A_0=A_0(\lambda)$ for any points $z(t),z(s)\in \lambda'D_j$, and for any two points in $\widetilde\Gamma$ situated on $T$. If $z(t)\in \lambda'D_i$, $z(s)\in \lambda'D_j$, $i\ne j$, then both parts of \eqref{chordarc} are comparable with $|x_i-x_j|$. To demonstrate this assertion we notice that the inequality $|z_1-z_2|>(\lambda-\lambda')(r_i+r_j)$, $z_1\in \lambda'D_i$, $z_2\in \lambda'D_j$, implies the relations
$|z_1-z_2|\asymp|x_i-x_j|$ and
$$
t-s\le C(r_i+|a_i-b_j|+r_j)\le C'|x_i-x_j|.
$$
Similar arguments yield \eqref{chordarc} when one point is situated on $T$.
\end{proof}

{\bf Open question.} It is not clear if the theorem is true or not when $\la=1$.

\subsection{Cauchy independence of families of measures}\label{inf} We use the results in the previous subsection to investigate the property of measures described below.

We call a finite Borel measure with compact support in the complex plane {\it a Cauchy operator measure} if the Cauchy operator $\mathcal{C}_\mu$ is bounded from $L^2(\mu)$ to itself wth norm at most $1$.

The first natural question is how to interpret the ``definition'' of $\mathcal{C}_\mu$ as
$$
\mathcal{C}_\mu f(z) = \int \frac{f(\xi)d\mu(\xi)}{\xi - z}.
$$
One of the ways is to consider the so-called $\varepsilon$-truncations, defined by
$$
\mathcal{C}^\varepsilon_\mu f(z) = \int_{\varepsilon < |\xi-z|<\varepsilon^{-1}} \frac{f(\xi)d\mu(\xi)}{\xi - z}.
$$
We now say that $\mathcal{C}_\mu$ is bounded as an operator from $L^2(\mu)$ to itself if the $\varepsilon$-truncations are bounded from $L^2(\mu)$ to itself uniformly in $\varepsilon$. Moreover, by the norm of $\mathcal{C}_\mu$ we understand the $\sup_{\varepsilon>0} \|\mathcal{C}^\varepsilon_\mu\|_{\mu}=:\|\mathcal{C}_\mu\|_{\mu}$\,, where $\|\mathcal{C}_\mu^\e\|_{\mu}$ is the norm of $\mathcal{C}_\mu^\e$ as an operator from $L^2(\mu)$ to itself. We encourage the reader to look for other interpretations in \cite{NTrV1}, \cite{Tolsa-book} and \cite{Volberg}.

The following important fact (which we will repeatedly use) demonstrates the connection between the analytic capacity and boundedness of the Cauchy operator \cite{Tolsa-sem, To2, Tolsa-book, Volberg}: for every compact set $F$ in $\C$,
\begin{equation}\label{gop}
\gamma(F) \asymp \sup\{\|\mu\|: \supp\mu\subset F,\ \mu\in\Sigma,\ \|\mathcal{C}_\mu\|_{\mu}\le1\},
\end{equation}
where $\Sigma$ is the class of nonnegative Borel measures $\mu$ such that $\mu(D(x,r))\le r$ for every disc $D(x,r):=\{z\in\C:|z-x|< r\}$.

We call a collection $\{\mu_j\}$ of finite positive Borel measures with compact supports {\it  $C$-Cauchy independent measures} if a) $\|\mathcal{C}_{\mu_j}\|_{\mu_j} \le 1$ (Cauchy operator measures) and b) $\|\mathcal{C}_\mu\|_\mu \le C<\infty$ for $\mu=\Sigma_j \mu_j$.
We will call such collection {\it Cauchy independent} if it is $C$-Cauchy independent for some finite $C$.

The family $\{\mu_j\}$ can be finite or infinite. Two Cauchy operator  measures are always Cauchy independent with an absolute constant $C$. A short proof of this nontrivial fact is given in \cite[Proposition~3.1]{NToV1}. So, a finite family is always Cauchy independent for a sufficiently large constant $C$. But our main interest is in situations, when infinite families are independent (or when $C$ is independent of the number of measures). The main result is the following

\begin{theorem}
\label{main1}
Suppose that $\la>1$, and measures $\mu_j$ are supported on compact sets $E_j$ lying in discs $D_j$ such that $\la D_j$ are disjoint. We also assume that measures $\mu_j$ are {\it extremal} in the following sense: $\|\mathcal{C}_{\mu_j}\|_{\mu_j}\le 1$ and $\|\mu_j\|\asymp \gamma(E_j)$ with absolute comparison constants.
Let $\mu=\sum_j \mu_j$ and $E=\cup E_j$. Then this family is Cauchy independent if and only if for any disc $B$,
\begin{equation}\label{mainc}
\mu(B) \le C_0 \gamma (B\cap E)\,.
\end{equation}
\end{theorem}
{\bf Remarks.}
1. In Section \ref{m1} we prove that the condition \eqref{mainc} with any disc $B$ is necessary for the bound $\|\mathcal{C}_\mu\|_\mu\le C$ without any additional assumptions on the structure of $\mu$. Example \ref{ex1} given in Section \ref{sh} shows that this condition alone is not sufficient even if $\mu$ consists of countably many pieces $\mu_j$, and each of $\mu_j$ gives a bounded Cauchy operator with a uniform bound. Thus, additional conditions on the structure of $\mu$ are needed. The example of such assumptions on $\mu$ which seems reasonable is given in Theorem \ref{main1}, where supports of $\mu_j$ are located in separated discs.
\smallskip

2. In general one cannot discard the requirement that the measures $\mu_j$ in Theorem \ref{main1} are extremal -- see Example \ref{ex2} in Section \ref{sh}.

3. On pp. 125--129, 135--146 of the paper of Tolsa \cite{Tolsa-sem}  it is proved that under the conditions of Theorem \ref{main1}, there exists a piece of measure $\mu$, namely $\mu':= \chi_{E'}\cdot \mu$, $E'\subset E$, such that $\mu'(E)\ge c\, \mu(E)$, and
$\|\mathcal{C}_{\mu'}\|_{\mu'} \le C<\infty$, where $c>0$ and $C$ are  constants depending only on parameters in Theorem \ref{main1}. This fact is far from being trivial, it is used in \cite{Tolsa-sem}  to approach Painlev\'e's conjecture. In other words, to prove that under the assumptions of Theorem \ref{main1} a ``good portion" of $\mu$ generates a bounded Cauchy operator is a highly non-trivial problem. It is remarkable that the whole measure $\mu$ has, in fact, such a property.
\smallskip

As a corollary of Theorem \ref{main1} we derive the following independence theorem.

\begin{theorem}
\label{main2}
Let $\mu_j$ be measures supported on compacts $E_j$ lying in discs $D_j$ such that $\la D_j$ are disjoint $(\la>1)$, and let $\|\mathcal{C}_{\mu_j}\|_{\mu_j} \le 1$. If for any disc $B$,
\begin{equation}\label{almadd}
\sum_j \gamma(B\cap E_j) \le C_1 \gamma(B\cap E),\quad E=\bigcup E_j,
\end{equation}
then the norm $\|\mathcal{C}_\mu\|_\mu$ is bounded, and the bound depends only on a comparison constant $C_1$ and on $\la$. Here, as above, $\mu=\Sigma_j \mu_j$.
\end{theorem}

Note that unlike Theorem \ref{main1}, this result does not need the additional condition that measures $\mu_j$ are extremal. On the other hand, \eqref{almadd} is not necessary for the boundedness of $\mathcal{C}_\mu$. For example, if $\mu_j$ are just 2-dimensional Lebesgue measures on the squares $E_j$, $j=1,2,\dots$, defined below in Example \ref{ex1}, the operator $\mathcal{C}_\mu$ is obviously bounded, but \eqref{almadd} does not hold.

Unlike Theorem \ref{superth}, Theorems \ref{main1}, \ref{main2} do not have any assumptions on the location of discs $D_j$.
However, condition
\eqref{almadd} is not completely independent of a geometry of discs. Theorem \ref{superth} states that if $\lambda$-separated discs are situated along a chord-arc curve, then the almost additivity of the analytic capacity takes place. We are going to prove a statement which is converse in the following sense: almost additivity of analytic capacity in the form of inequality \eqref{almadd} together with certain additional assumptions imply that our discs have to ``line-up" along a good (Ahlfors-David regular) curve.

By $\mathcal{H}^1$ we denote the 1-dimensional Hausdorff measure. A set $G\subset \C$ is called Ahlfors-David (AD) regular if
$$
c\,r\le\mathcal{H}^1(G\cap B(x,r)) \le Cr, \quad x\in G,\ 0<r\le\diam G,
$$
with some positive constants $c,C$.

\begin{corollary}\label{geom}
Suppose that $\la$-separated ($\la>1$) discs $D_j=D(x_j,r_j)$ and subsets $E_j\subset D_j$ are such that (a) \eqref{almadd} holds, (b) $\g(E_j)\asymp r_j$, (c) the set $\mathcal T:=\bigcup_j\partial D_j$ is AD regular. Then there exists an AD regular curve which intersects all discs $D_j$.
\end{corollary}
Without any of assumptions (a) -- (c) Corollary \ref{geom} is incorrect -- see Proposition \ref{prop53}.

To prove Theorems \ref{main1}, \ref{main2}, we will need only the special case of Theorem \ref{superth} when all discs $D_j$ intersect a real line or a circle, that is the last assertion of Corollary \ref{circle}. In this case there is a short proof based only on some classical facts in complex analysis. We give this proof in the next Section \ref{almost}. Theorem \ref{main1} is proved in Section \ref{m1}, and Theorem \ref{main2} with Corollary \ref{geom} in Section \ref{m2}. In Section \ref{sh} we give the examples mentioned above. Section \ref{supth} contains the proof of Theorem \ref{superth} in the full generality, which is completely different from the proof in Section \ref{almost}. The main tool of this proof is  Melnikov--Menger's curvature of a measure. All necessary definitions are given in Section \ref{supth}. In the last Section \ref{q} we formulate an open question.

\section{Almost-additivity of analytic capacity:\\
string of beads attached to the real line}
\label{almost}

A result close to the theorem below for some special sets $\{E_j\}_{j=1}^\infty$ was proved (but not stated) in \cite{NV}. Here we use the approach via the Marcinkiewicz function, the approach in \cite{NV}  was a bit more complicated. Unlike \cite{NV}, we do not need any special size properties of these sets.

\begin{theorem}
\label{sup}
Let $D_j$ be discs such that $\la D_j \cap \la D_k = \emptyset$,  $j\not=k$, for some $\la>1$, and each disc $D_j$ has a non-empty intersection with the real line $\R$. Let $E_j\subset D_j$ be arbitrary compact sets. Then there exists a constant $c=c(\la)>0$, such that
$$
\gamma\big(\bigcup E_j\big) \ge c \sli_j \gamma(E_j).
$$
\end{theorem}

\begin{proof}
It is enough to prove the result for finite families of indices $j$. We first notice that  $\gamma_j := \gamma(E_j) \le \g(D_j)= r_j$, where $r_j$ is the radius of $D_j$.
Let $y_j$ be the center of the chord $\R\cap D_j$. Denote $\lambda' := \frac{1+\lambda}{2}$.
For each $j$  we draw a horizontal line segment $L_j\subset \lambda' D_j$ with center at $y_j$ and with the analytic capacity $b(\lambda)\gamma_j$, where $b(\la)=\frac{\sqrt{\lambda'^2-1}}4 $. Thus, the length $\ell_j$ of $L_j$ satisfies $\ell_j =4b(\lambda)\gamma_j\le r_j\sqrt{\lambda'^2-1} $. In particular, the whole segment $L_j$ lies in $\la' D_j$.

Next, let $f_j$ be the function that gives the capacity of $E_j$. Also let $\vf_j$ be the function that gives the capacity of $L_j$. Then we have
$$
\vf_j(z)=\int_{L_j}\frac{\vf_j(x)}{x-z}dx, \quad \int_{L_j} \vf_j(x)dx = b(\lambda)\gamma_j.
$$
Positive functions $\vf_j(x)$ have a uniform bound $\|\vf_j\|_{\infty} \le A$ with an absolute constant $A$. In particular, if $\mathcal{F}$  is any subset of indices $j$ we have
\begin{equation}
\label{img}
\bigg|\Im\, \sum_{j\in \mathcal{F}} \vf_j(z)\bigg| \le A \int_{\bigcup_{j\in \mathcal{F}}L_j}\frac{|\Im z|}{|t-z|^2}\,dt\le \pi A\,,\quad \forall z\in \mathbb{C}\,.
\end{equation}

\smallskip
{\bf Remark.} It is important here that the intervals $L_j$ are situated on the real line (or at least are not far away from $\R$). For any $M>0$ one can easily construct a chord-arc curve and discs centered on it such that the left hand side in \eqref{img} exceeds $M$. This is the obstacle for extension of these arguments to chord-arc curves.
\smallskip

Our next goal is to find a family $\mathcal{F}$ of indices and absolute positive constants $a_1$, $a_2$, such that the following two assertions hold:
\begin{gather}
\sum_{j\in \mathcal{F}} \gamma_j \ge a_1 \sum_j \gamma_j\,, \label{sglarge}\\
\sum_{j\in \mathcal{F}} |f_j(z)-b(\lambda)^{-1}\vf_j(z)| \le a_2\,,\quad \forall z\in \mathbb{C}\setminus\bigcup_{j\in \mathcal{F}}(E_j\cup L_j).\label{fg}
\end{gather}

 Let us finish the proof of the theorem, taken these assertions as granted (for a short while). Let $F:=\sum_{j\in \mathcal{F}} f_j$. Combining \eqref{img} and \eqref{fg} we get $|\Im\, F(z)| \le C_1(\la)$, $z\in\mathbb{C}\setminus(\bigcup_{j\in \mathcal{F}}E_j)$. Hence, the function $F_1:=e^{iF}-1$ is bounded in $\mathbb{C}\setminus(\bigcup_{j\in \mathcal{F}}E_j)$ by a constant $C(\la)$. Since $F(\infty)=0$, we have $|F_1'(\infty)|=|F'(\infty)|=\sum_{j\in \mathcal{F}} \gamma_j$. Thus,
$$
\gamma\bigg(\bigcup_{j\in \mathcal{F}} E_j\bigg) \ge \frac{1}{C(\la)} \sum_{j\in \mathcal{F}} \gamma_j\,.
$$
 Combine this with \eqref{sglarge}. We obtain, that
 $$
 \gamma\bigg(\bigcup_j E_j\bigg) \ge \gamma\bigg(\bigcup_{j\in \mathcal{F}} E_j\bigg) \ge a_3 \sum_j \gamma_j\,,
$$
and Theorem \ref{superth} would be proved. So we are left to chose the family $\mathcal{F}$ such that \eqref{sglarge}, \eqref{fg} hold.

By the Schwartz lemma in the form we borrow from \cite[p.~12--13]{G}, we have
\begin{equation}
\label{S}
|f_j(z)-b(\lambda)^{-1}\vf_j(z)| \le \frac{A r_j \gamma_j}{\dist(z, E_j\cup L_j)^2}\,,\quad z\notin E_j\cup L_j\,.
\end{equation}
Denote
$$
Q_i:=\la'D_i\,, \quad g_i:= \sum_{j:\,j\neq i} \frac{ r_j \gamma_j}{D(Q_j, Q_i)^2}\,,
$$
where $D(Q_i, Q_j):= \text{dist}(Q_i, Q_j)$.

\smallskip
{\bf Remark.} We do not need this, but for the sake of explanation, let us  define a function $g=\sum g_j \chi_{{Q_j}\cap\R}$. This function is often called a Marcinkiewicz function.  The main trick with Marcinkiewicz functions is to integrate them with respect to a suitable measure. What in fact happens next is that  we integrate it with respect to Lebesgue measure on $\R$.
\smallskip

The important point is that we can estimate $\sum_i g_i\gamma_i$. In fact,
\begin{align*}
\sum_i g_i\gamma_i &=\sum_i\gamma_i \sum_{j:\,j\neq i} \frac{r_j\gamma_j}{D(Q_j, Q_i)^2}= \sum_j r_j \gamma_j  \sum_{i:\,i\neq j} \frac{\gamma_i}{D(Q_i, Q_j)^2} \\
&\le \sum_j r_j \gamma_j  \sum_{i:\,i\neq j} \frac{r_i}{D(Q_i, Q_j)^2} \le A_0 \sum_j r_j \gamma_j r_j^{-1}=A_0\sum_j\gamma_j\,.
 \end{align*}
 In the last estimate we used that
 $$
 \sum_{i:\,i\neq j} \frac{r_i}{D(Q_i, Q_j)^2} \le A_1\int_{t: |t-y_j|\ge r_j} \frac{1}{|t-y_j|^2} \,dt \le \frac{2A_1}{r_j}\,,\quad A_1=A_1(\la),
 $$
since the length of $\R\cap Q_j$ exceeds $c(\la)r_j$, and $D(Q_i, Q_j)\ge c'(\la)|t-y_j|$, $t\in\R\cap Q_i$.

Now we apply the Tchebychev inequality. Denote $I^*:=\{i:\, g_i > 10 A_0\}$, $I_*:= \{i:\, g_i \le10 A_0\}$.  We immediately see that
\begin{equation}
\label{I}
\sum_{j\in I_*} \gamma_j \ge \frac{9}{10} \sum_j\gamma_j\,.
\end{equation}

Obviously, by \eqref{S} for every index $i$ we have
$$
\sum_{j:\, j\neq i} |f_j(z)-b(\la)^{-1}\vf_j(z)| \le A g_i\,,\quad z\in  Q_i\,.
$$
This estimate and the choice of $I_*$ imply that
$$
\sum_{j:\, j\neq i,\ j\in I_*} |f_j(z)-b(\la)^{-1}\vf_j(z)| \le C(\la) g_i\le 10A_0A\,,\quad \forall i\in I_*\,,\ \forall z\in Q_i\,.
$$
But all functions $|f_i|, |\vf_i|$ are bounded by $1$ in $\mathbb{C}\setminus(E_i\cup L_i)$. Therefore, the last inequality implies the estimate
\begin{multline}\label{sumfg}
\sum_{j:\,  j\in I_*} |f_j(z)-b(\la)^{-1}\vf_j(z)| \le 10A_0A+b(\la)^{-1}+1=: a_2\,\\
 \forall z\in Q_i\setminus(E_i\cup L_i)\,,\ \forall i\in I_*.
\end{multline}
The function $\sum_{j\in I_*}( f_j-b(\la)^{-1}\vf_j)$ is analytic in $\mathbb{C}\setminus\bigcup_{i\in I_*}(E_i\cup L_i)$ and vanishes at infinity. Therefore, \eqref{sumfg} implies \eqref{fg} if we put $\mathcal{F}:=I_*$. Assertion \eqref{sglarge} is proved in \eqref{I}, and the proof of Theorem \ref{sup} is completed.
\end{proof}

\begin{corollary}\label{circle1}
The statement of Theorem \ref{sup} remains true if discs $D_j$ intersect  a circle instead of the real line.
\end{corollary}
\begin{proof}
It is sufficient to consider the unit circle.
There are at most $K$ discs $D_j$ with radii $r_j\ge1/30$, intersecting the unit circle $\Gamma$, where $K$ is an absolute constant. Hence, we may assume that $r_j<1/30$. Moreover, as in Corollary \ref{circle}, we may restrict ourself by discs intersecting the left semicircle $T$. Let $h(z):=i\frac{1+z}{1-z}$ be the conformal mapping of $\Gamma$ onto the real line. Let $E$ be a compact subset of $G^0:=\{z:\dist(z,T)<1/10\}$, and $\mathcal E:=h(E)$. We prove that
\begin{equation}\label{eqlast}
\g(E)\asymp\g(\mathcal E)
\end{equation}
with absolute constants of comparison. Certainly, this relation is a consequence of the general result by Tolsa \cite{To3} about stability of the analytic capacity under bilipschitz maps. But we prefer a direct elementary proof (which possibly is not new) without Tolsa's very difficult result.
Pick any function $f(z)$ such that $f(z)$ is holomorphic and $|f(z)|<1$ in $\C\setminus E$. Define the sets
$$
G:=\{z:\dist(z,T)<1/5\},\quad \mathcal G:=h(G),\quad \mathcal G^0:=h(G^0),
$$
and the function $F(w):=f(g(w))g'(w)$, where $g(w)=\frac{w-i}{w+i}$ is the inverse of $h$. Clearly, $|F(w)|\le C_1$ as $w\in \mathcal G\setminus \mathcal E$, and the length of $\partial \mathcal G$ does not exceed $C_2$, where $C_1,C_2$ are absolute constants. Fix $w\in \mathcal G^0\setminus \mathcal E$, and let $L^0$ be a closed curve in $\mathcal G^0\setminus \mathcal E$ which encloses $\mathcal E$ but not $w$, and is oriented in such a way that $w$ is ``on the left''. If $\partial \mathcal G$ is oriented in the same way, by the Cauchy formula we have
$$
F(w)=\frac1{2\pi i}\int_{\partial \mathcal G}\frac{F(\xi)}{\xi-w}\,d\xi+\frac1{2\pi i}\int_{L^0}\frac{F(\xi)}{\xi-w}\,d\xi=:F_1(w)+F_2(w).
$$
Since $|\xi-w|$ exceeds an absolute constant as $\xi\in\partial \mathcal G$, $|F_1(w)|<C$ as $w\in \mathcal G^0\setminus \mathcal E$. But $F$ is bounded in this domain as well. Hence, $F_2(w)$ is bounded in $\mathcal G^0\setminus \mathcal E$ and holomorphic in $\C\setminus \mathcal E$. By the maximum principle, $F_2(w)$ is holomorphic and bounded by an absolute constant in $\C\setminus \mathcal E$. Let $L$ be a contour in $\mathcal G\setminus \mathcal E$ inclosing $\mathcal E$. Since $F_1(w)$ is holomorphic in $\mathcal G$, we have $\int_{L}F_1(w)\,dw=0$. Hence,
$$
\int_{L}F_2(w)\,dw=\int_{L}F(w)\,dw=\int_{L}f(g(w))\,dg(w)=\int_{g(L)}f(z)\,dz.
$$
Thus, $\g(E)\le C\g(\mathcal E)$. The similar arguments yield the inverse inequality, if we start with the compact $\mathcal E$ (we may assume that $\mathcal E\subset \mathcal G^0$). Therefore, \eqref{eqlast} is proved. Corollary \ref{circle1} is a direct consequence of \eqref{eqlast} and Theorem \ref{sup}.
\end{proof}

\section{Proof of Theorem \ref{main1}}
\label{m1}

\subsection{Necessity of the condition \eqref{mainc}}\label{s31}
Suppose that $\|\mathcal C_\mu\|_\mu\le C<\infty$ and $\supp\mu\subset E$. One can easily see that $\|\mathcal C_{\mu|B}\|_{\mu|B}\le C<\infty$ for any disc $B$. Moreover, boundedness of $\mathcal C_\mu$ implies that $\alpha\mu\in\Sigma$ with $\alpha$ depending only on $C$, see for example \cite{David-LNM}. Thus, the measure $c\mu|B$, $c=c(C)>0$, participates in the right hand side of \eqref{gop} with $F=B\cap E$, and we get \eqref{mainc}.

\subsection{Sufficiency of the condition \eqref{mainc}}\label{s32}

The following result was proved in \cite{NToV1}, although was not formulated explicitly (see the last three pages of Section~3 in \cite{NToV1}).

\begin{theorem}\label{comp}
Suppose that $\{D_j\}$ are discs on the plane and the dilated discs $\la D_j$, $\la>1$, are disjoint. Let $\nu, \sigma$ be two positive measures supported in $ \bigcup_j D_j$ such that $c_1\nu(D_j)\le \sigma(D_j) \le c_2 \nu(D_j)$, $0<c_1<c_2<\infty$. Suppose also that the Cauchy operators $\mathcal C_{\sigma_j}$, $\sigma_j=\sigma|_{D_j}$, are uniformly bounded. Then if $\nu$ is a Cauchy operator measure, then $\a\sigma$ is also a Cauchy operator measure with a constant $\a$ depending only on $c_1, c_2$, and $\la$.
\end{theorem}

We need some preliminary constructions and notations. Here is an easy lemma.
\begin{lemma}\label{crossB}
For any circle $T$ and any disc $B$,
$$
\gamma(T \cap B) \asymp \mathcal{H}^1(T\cap B)
$$
with absolute constants of comparison.
\end{lemma}

Now we define new $L_j$. We need a number $N=N(\la)$, which is defined as follows. Recall that $\la>1$ and $\la'= \frac{1+\la}{2}$. Let a disc $D$ with radius $r$ be given. We place the circle of radius $A_\la r$, $A_\la :=\min(1,\la'-1)/1000$, in the center of $D$, and $N$ circles that touch $\partial(\la'D_j)$ from within, of the same radius as the first circle, and on equal distance from each other. We also require that circles do not intersect. By $L$ we denote the union of all circles. Let $N$ be a minimal integer such that the following holds.

\begin{it}
If a disc $B$ intersects $D$ and $\C \setminus \la D$, then at least one circle from $L$ lies inside $B$.
\end{it}

Clearly, such $N=N(\la)<\infty$ exists, and this property remains valid if we reduce the radii of circles in $L$. Since
$\gamma(\mbox{circle}) \asymp\mathcal{H}^1(\mbox{circle})$, we have the obvious lemma:
\begin{lemma}
For the set $L$ defined above, $\gamma(L)\asymp \mathcal{H}^1(L)$, where the comparison constants can depend only on $N$.
\end{lemma}

Let $L_j$ be the union of circles associated with $D_j$. We have chosen the number of circles in each $L_j$, but we have a freedom to choose their size. We redefine the radius so that
\begin{equation}\label{eq31}
\mathcal{H}^1(L_j) = A_\la(N+1)\gamma(E_j).
\end{equation}
Then, in particular, $\mathcal{H}^1(\mbox{one circle}) = \frac{1}{N+1} \mathcal{H}^1(L_j) =A_\la\gamma(E_j) \le A_\la r_j$, since $\gamma(E_j) \le \gamma(D_j) =r_j$.

We need the following lemma.
\begin{lemma}
\label{discinside}
Fix an index $j$. Let $B$ be a disc such that at least one circle from $L_j$ lies inside $B$. Then $\gamma(L_j)\asymp \gamma(L_j\cap B)$ with  constants depending only on $\la$. In particular this is true if $D_j\subset B$, or if $B$ intersects $D_j$ and $\C\setminus \la D_j$.
\end{lemma}
\begin{proof}
Indeed, by semiadditivity of $\gamma$ we have $\gamma(L_j)\le A (N+1) \gamma(\mbox{central circle}) \le A(N+1)\gamma(L_j \cap B)$.
\end{proof}
\begin{lemma}\label{cutfrominsidediscs}
For any disc $B$ the following relation holds with  constants depending only on $\la$:
$$
\gamma\big(\bigcup_{j:D_j\subset B} L_j\big)\asymp \gamma\big(\bigcup_{j:D_j\subset B} L_j \cap B\big).
$$
\end{lemma}
\begin{proof}
By semiadditivity of $\gamma$,
$$
\gamma\big(\bigcup_{D_j\subset B} L_j\big) \le A \bigg(\gamma\big(\bigcup_{\la'D_j\subset B}L_j\big) + \gamma\big(\bigcup_{D_j\subset B,\ \la'D_j\not\subset B}L_j\big)\bigg).
$$
The first term is the same as $\gamma(\bigcup_{\la'D_j\subset B}L_j\cap B)$. For the second, we notice that $L_j\cap B\subset \la'D_j$. Since discs $\frac{\la}{\la'}\la' D_j$ are pairwise disjoint, we may apply Corollary \ref{circle1} with just a new dilation constant $\la_{new}:= \frac{\la}{\la'}$. Thus,
\begin{align*}
\gamma\big(\bigcup_{D_j\subset B,\ \la'D_j\not\subset B} L_j\cap B\big) &\ge c\sli_{D_j\subset B,\ \la'D_j\not\subset B} \gamma(L_j\cap B)\\
&\ge c_1 \sli_{D_j\subset B,\ \la'D_j\not\subset B} \gamma(L_j) \ge c_2 \gamma\big(\bigcup_{D_j\subset B,\ \la'D_j\not\subset B} L_j\big),
\end{align*}
which finishes the proof. The second inequality uses Lemma \ref{discinside}.
\end{proof}

For a given disc $B$ denote by $\J=\J(B)$ the set of indices
$\J:=\{j: D_j\cap B \not= \emptyset \ \mbox{and} \ D_j\not\subset B\}$.

\begin{lemma}
\label{cutfromboundarydiscs}
Suppose that a disc $B$ intersects more than one $D_j$. Then
$$
\gamma\big(\bigcup_{\J} L_j\big)\asymp \gamma\big(\bigcup_{\J} L_j \cap B\big).
$$
with comparison constants depending only on $\la$.
\end{lemma}
\begin{proof}
Here again we will use Corollary \ref{circle1} of Theorem \ref{sup}. Since $B$ intersects more than one $D_j$,  it cannot be contained in $\la D_j$,  $j\in \J$.  Thus, it contains at least one circle from $L_j$ for each $j\in \J$ (it follows from the choice of $N$).  Call this circle $C_j$. We  apply
Corollary \ref{circle1} to $D_j$ with dilation constant $\la$ to get the estimate
$$
\gamma\big(\bigcup_{\J} L_j \cap B)\ge c \sli_{\J}\gamma(L_j\cap B)
\ge\sli_{\J} \gamma(\mbox{$C_j$}) \ge c_1 \sli_{\J} \gamma(L_j)\ge c_2 \gamma(\bigcup_{\J}L_j),
$$
which finishes the proof.
\end{proof}

Finally, we need the following notation. Fix a disc $B$. Denote
$$
F_j=\begin{cases} E_j, & D_j\subset B \\
\emptyset, & D_j\not\subset B. \end{cases}, \quad F=\bigcup F_j.
$$
{\bf Remark.}
A disc $B$ will be free to change in what follows. The constants in further inequalities will never depend on $B$.
\smallskip

Our next goal is to prove that under assumptions of Theorem \ref{main1}, the inequality
$$
\gamma\big(\bigcup L_j \cap B\big) \ge c\sli \gamma(L_j\cap B)
$$
holds with a universal constant $c$ (universality means that $c$ will not depend on the disc $B$). We need the following two lemmas.

 We fix a small positive absolute constant $\varepsilon$. The choice of smallness will be clear from what follows.
\begin{lemma}[The first case]\label{firstcase}
Suppose that $\gamma(F)\le \varepsilon \gamma(E\cap B)$. Then there exists a constant $c$, that can depend only on $N$, $\varepsilon$, and other universal constants, such that
$$
\gamma\big(\bigcup L_j \cap B\big) \ge c\sli \gamma(L_j\cap B).
$$
\end{lemma}

\begin{proof}
Suppose that $B$ intersects only one $\la D_j$. Then the $\bigcup$ and the $\sum$ have only one term, and there is nothing to prove. So, we can assume that $B$ intersects at least two of $\la D_j$'s.
Notice also that by this assumption, by the fact that $\la D_i$ are pairwise disjoint, and by the choice of $N$, if $B$ intersects $D_j$ then at least one circle from $L_j$ lies inside $B$. Let $\J$ be as in Lemma \ref{cutfromboundarydiscs}.
Using Lemma \ref{discinside} and Corollary \ref{circle1} we get

\begin{equation}
\label{zero}
 \sli_{\J} \gamma(L_j)\le A_1 \sli_{\J}\gamma(L_j\cap B) \le A_2 \gamma(\bigcup_{\J}L_j \cap B)\,.
 \end{equation}

On the other hand, by the assumption of Theorem \ref{main1},
\begin{equation}\label{star}
\sli_{D_j\subset B}\gamma(L_j) \le C\sli_{D_j\subset B}\gamma(E_j) \le C'\sli_{D_j\subset B} \mu_j(D_j)\le C'\mu(B) \le C'C_0 \gamma(E\cap B).
\end{equation}
Also with an absolute constant $A$,
$$
\gamma(E\cap B) \le A \bigg(\gamma(F)+\gamma\big(\bigcup_{\J} E_j \cap B\big)\bigg) \le \varepsilon A \gamma(E\cap B) + A\gamma\big(\bigcup_{\J} E_j \cap B)\big).
$$
Thus, if $\varepsilon$ is small enough (notice that the smallness depends only on $A$), we have
\begin{equation}
\label{twostar}
\gamma(E\cap B)\le C \gamma\big(\bigcup_{\J} E_j \cap B)\big).
\end{equation}
Therefore, combining \eqref{star},  \eqref{twostar}, and \eqref{zero}, we obtain
\begin{equation}\label{3star}
\begin{split}
\sli_{D_j\subset B}\gamma(L_j)&\le C \gamma\big(\bigcup_{\J} E_j \cap B)\big) \le C_1 \sli_{\J}\gamma(E_j\cap B)\\
& \le C_2 \sli_{\J}\gamma(L_j) \le C_3 \gamma\big(\bigcup_{\J} L_j\cap B\big).
\end{split}\end{equation}

Now combine \eqref{zero} and \eqref{3star} to get
\begin{equation}
\label{12a}
\gamma\big(\bigcup_{} L_j\cap B\big) \ge \gamma\big(\bigcup_{\J} L_j\cap B\big) \ge  c \sli_{D_j\subset B}\gamma(L_j)  + c \, \sli_{\J} \gamma(L_j)=c \sli_{D_j\cap B\ne\emptyset}\gamma(L_j)\,.
\end{equation}
Obviously,
\begin{equation}
\label{12b}
\gamma\big(\bigcup_{} L_j\cap B\big) \ge \gamma\big(\bigcup_{\J_1} L_j\cap B\big),\quad \J_1:=\{j:D_j\cap B=\emptyset,\ L_j\cap B\not=\emptyset\}.
\end{equation}
For $j\in\J_1$ we again consider the new dilation constants $\la_{new}:=\frac{\la}{\la'}$, and
discs $D_j':=\la' D_j$. The  discs $\la_{new}D_j'$, $j\in\J_1$, are disjoint, and $D_j'$ intersects $B$ for $j\in \J_1$. By Corollary \ref{circle1}  applied to $L_j\cap B$, $j\in \J_1$, playing the roles of $E_j$, we get
\begin{equation}
\label{12c}
\gamma\big(\bigcup_{\J_1} L_j\cap B\big) \ge c\sli_{\J_1}\gamma(L_j\cap B)\,.
\end{equation}
The combination of \eqref{12a}--\eqref{12c}  finishes the proof.
\end{proof}

\begin{lemma}[The second case]
\label{secondcase}
Suppose that $\gamma(F)\ge \varepsilon \gamma(E\cap B)$ with $\varepsilon$ from the previous lemma. Then there exists a universal constant $c$ such that
$$
\gamma\big(\bigcup L_j \cap B\big) \ge c\sli \gamma(L_j\cap B).
$$
\end{lemma}
\begin{proof}
By Theorem \ref{sup} or rather Corollary \ref{circle1} we need only to prove the inequality
\begin{equation}
\label{inside}
\gamma\big(\bigcup_{D_j\subset B} L_j\cap B\big) \ge c \sli_{D_j\subset B}\gamma(L_j\cap B).
\end{equation}
Using the assumption of our lemma as well as the conditions \eqref{mainc} and $\|\mu_j\|\asymp\g(E_j)$ of Theorem~\ref{main1}, we get
\begin{equation}\label{reduction}
\gamma(F)\ge \varepsilon \gamma (E\cap B) \ge \varepsilon c\mu(B) \ge \varepsilon c \sli_{D_j\subset B}\mu_j(B) \ge \varepsilon c' \sli_j \gamma(F_j).
\end{equation}

By $\nu$ we denote the measure on $F$ participating in \eqref{gop} for which $\|\nu\| \asymp \gamma(F)$. Denote $d\nu_j = \chi_{F_j}d\nu$. Then $\mathcal{C}_{\nu_j}$ is bounded on $L^2(\nu_j)$ (with norm at most $1$), and \eqref{gop} yields the estimate
$$
\|\nu_j\| \le C \gamma(F_j) \le C_1 \gamma(L_j) \le C_2 \mathcal{H}^1(L_j) =: C_2 \ell_j.
$$

We call $j$ {\it good} if $D_j\subset B$ and $\|\nu_j\| \ge \tau \ell_j$. The choice of $\tau$ will be clear from the next steps. However, we want to emphasize now that this choice will be universal. By \eqref{reduction} we have:
\begin{align*}
\varepsilon c' A^{-1}\sli \gamma(F_j) &\le A^{-1}\gamma(F) \le \|\nu\| = \sli \|\nu_j\| \\
&\le C_2 \sli_{j \text{ is good}} \ell_j + \tau \sli_{D_j\subset B} \ell_j \le C_2\sli_{j \; \text{is good}} \ell_j +C_3\tau\sli \gamma(F_j)
\end{align*}
(in the last inequality we use \eqref{eq31}). Therefore,
\begin{equation}
\label{goodj}
\sli\ci{j \; \text{is good, }D_j\subset B} \ell_j \ge c \sli \gamma(F_j)\ge c_1 \sli_{D_j\subset B} \ell_j.
\end{equation}
Actually, $\tau$ is chosen exactly here. We see that it indeed depends only on universal constants such as $A$ and $\varepsilon$. Recall that $C_j$ denotes the central circle of each $L_j$. We set
$$
d\sigma_g := \sli_{j \; \text{is good, }D_j\subset B} \chi_{C_j} d\mathcal{H}^1, \quad d\nu_g := \sli_{j \; \text{is good, }D_j\subset B} d\nu_j\,.
$$

Then for good $j$, $\sigma_g(D_j) =\mathcal{H}^1(C_j)\asymp \mathcal{H}^1(L_j) =\ell_j\asymp \nu_g(D_j)$. In the last relation the comparison constants can depend on previous universal constants and $\tau$. Operators $\mathcal{C}_{\sigma_g|D_j}$ are uniformly bounded. By the choice of $\nu$, the operator $\mathcal{C}_{\nu_g}$ is bounded as well with norm at most $1$. Thus, we may apply Theorem \ref{comp} and conclude that $\mathcal{C}_{\sigma_g}$ is also bounded with a certain absolute bound of the norm. Therefore, using \eqref{goodj}, we get
$$
\gamma\big(\bigcup_{D_j\subset B} L_j\big) \ge \gamma\big(\bigcup_{j \text{ is good, }D_j\subset B} L_j\big) \ge c \|\sigma_g\| \ge c_1 \sli_{j \text{ is good, }D_j\subset B} \ell_j \ge c_2 \sli_{D_j\subset B} \ell_j.
$$

In Lemma \ref{cutfrominsidediscs} we have proved that
$$
\gamma\big(\bigcup_{D_j\subset B} L_j\big) \asymp \gamma\big(\bigcup_{D_j\subset B} L_j\cap B\big).
$$
Moreover, for every $j$ such that $D_j\subset B$, we have
$$
\ell_j=\mathcal{H}^1(L_j)\asymp \mathcal{H}^1(L_j\cap B)\asymp \gamma(L_j\cap B).
$$
Thus, we obtain \eqref{inside}, and Lemma \ref{secondcase} is proved.
\end{proof}

The main Theorem of \cite{NV} says:

\begin{theorem}
\label{h1}
Let $L\subset\R^2$, be a compact set of positive and finite Hausdorff measure $\mathcal{H}^1$, and let $\sigma=\mathcal{H}^1|L$. Then $\mathcal{C}_{\sigma}$ is bounded if and only if there exists a finite constant $C_0$ such that $\sigma(B\cap L) \le C_0 \gamma(B\cap L)$ for any disc $B$.
 \end{theorem}

Starting with  the main assumption of Theorem \ref{main2} (the inequality $\mu(B)\le C_0 \gamma(B\cap E)$ for any disc $B$) we proved in Lemmas \ref{firstcase}, \ref{secondcase} that the  uniform in $B$ almost-additivity of $\gamma$ holds for the union of all sets $\{L_j\cap B\}$. Namely, we proved that the following holds for any $B$ with uniform positive $c_2$:
\begin{equation}
\label{Lj}
\gamma(B\cap L) \ge c_1 \sum_j\gamma( B\cap L_j) \ge c_2\sum_j \sigma (B\cap L_j) =c_2\sigma(B\cap L)\,,
\end{equation}
where $\sigma:= \mathcal{H}^1|L$.
Hence the measure $\sigma$ satisfies Theorem \ref{h1}. So the boundedness of Cauchy operator on the union of circles is obtained. The measures $\sigma|L_j$ and $\mu_j$ are supported on $\la' D_j$, the discs $\frac{\la}{\la'}\cdot(\la' D_j)$ are disjoint, and  $\sigma(L_j)\asymp\mu_j$ (see \eqref{eq31}). We may apply Theorem \ref{comp} again to establish the boundedness of $\mathcal{C}_{\mu}$ in $L^2(\mu)$.

\section{Proof of Theorem \ref{main2} and Corollary \ref{geom} }\label{m2}

\begin{proof}[Proof of Theorem \ref{main2}] We would like to reduce the assumptions of Theorem \ref{main2} to conditions of Theorem \ref{main1}.
For every $E_j$, choose an ``extremal'' measure $\mu_j'$, supported on $E_j$ and such that $\|\mathcal{C}_{\mu_j'}\|_{\mu_j'} \le 1$, and $c'_1\mu_j'(E_j)\le \gamma(E_j)\le c'_2 \mu_j'(E_j)$, $0<c'_1<c'_2<\infty$ (the existence of such measures $\mu_j'$ follows from \eqref{gop}).

Now let us consider new measures
$$
\wt\mu_j:= \mu_j +\mu_j'\,,\quad \wt\mu=\sum_j \wt\mu_j\,.
$$
We prove that these new measures satisfy all assumptions of Theorem \ref{main1} (up to an absolute constant).
Indeed, we saw in Section \ref{s31} that the assumption $\|\mathcal{C}_{\mu_j}\|_{\mu_j} \le 1$ implies the inequality $\mu_j(E_j)\le C\gamma(E_j)$. Hence, the new measures $\wt\mu_j$ also satisfy the ``extremality" condition ($\mu_j$ did not satisfy it in general):
$$
c_1\wt\mu_j(E_j)\le \gamma(E_j)\le c_2 \wt\mu_j(E_j),\quad 0<c_1<c_2<\infty.
$$
Moreover, since two Cauchy operator measures are always Cauchy independent (see Section \ref{inf}), $\|\mathcal{C}_{\wt\mu_j}\|_{\wt\mu_j} \le C$. Now the relation \eqref{gop} implies that $\wt\mu(B\cap E_j) = \wt\mu_j(B\cap E_j)\le C \gamma(B\cap E_j)$ for every disc $B$. Therefore, by \eqref{almadd} we have
$$
\wt\mu(B) = \sum\wt\mu(B\cap E_j)\le C \sli \gamma(B\cap E_j) \le C_0 \gamma(B\cap E),\quad C_0=C_0(C_1).
$$
where the latter is the condition \eqref{mainc} of Theorem \ref{main1} for $\wt\mu$.

We apply Theorem \ref{main1} to $c\wt\mu_j$ and $c\wt\mu=c\sum_j\wt\mu_j$ with a sufficiently small absolute constant $c>0$, and get that
$\|\mathcal{C}_{\wt\mu}\|_{\wt\mu}\le C(C_1, \la)$. But $\mu=\sum_j\mu_j$ is a part of $\mu$, and hence $\|\mathcal{C}_{\mu}\|_{\mu}\le C(C_1, \la)$.
\end{proof}

\begin{proof}[Proof of Corollary \ref{geom}]
Recall the reader that above we built $L_j$ for each $D_j$, and each $L_j$ contains a certain ``central circle". Let $\mathcal T'$ be the union over $j$ of central circles in $L_j$. By \eqref{eq31}, the radii of these circles are comparable with $r_j$. Since $\mathcal T$ is AD regular and discs $D_j$ are $\la$-separated, $\mathcal T'$ is AD regular as well (with another constants $c,C$). We have proved in the previous section that the Cauchy operator $\mathcal{C}_{\sigma}$, $\sigma:= \mathcal{H}^1|L$, is bounded from $L^2(\sigma)$ to itself. Hence, the operator $\mathcal{C}_{\sigma'}$, $\sigma':= \mathcal{H}^1|\mathcal T'$, is bounded too. By the theorem of Mattila, Melnikov and Verdera \cite{MMV}, the set $\mathcal T'$ is contained in an AD regular curve.
\end{proof}

\section{Examples}\label{sh}

We saw in Section \ref{m1} that the condition
\begin{equation}\label{mug}
\mu(B)\le C_0 \gamma(B\cap E) \text{\ \ for every disc } B
\end{equation}
is necessary for the boundedness of the Cauchy operator $\mathcal{C}_\mu$ with any Borel measure $\mu$. It is not difficult to see that this condition alone is not enough for the boundedness of $\mathcal{C}_\mu$. Indeed, let $\mu_n^{1/4}$ be the probability measure uniformly distributed on the set $E_n^{1/4}$ defined in Introduction. Let $\mu^{1/4}$ be the weak limit of some weakly convergent subsequence $\{\mu_{n_k}^{1/4}\}$, $E^{1/4}=\bigcap E_n^{1/4}$, $E$ is the initial unit square, and $\mu:=\mu^{1/4}+\mathcal H^2|E$. Then $\mu$ satisfies \eqref{mug}, but $\mathcal{C}_\mu$ is unbounded -- see for example \cite{MT, MTV}. We are going to demonstrate more: in general the condition \eqref{mug} is not sufficient for the boundedness even if $\mu$ consists of  countably many pieces, and each of them gives a bounded Cauchy operator.

\begin{example}\label{ex1} There exists a family of measures $\{\mu_{j}\}_{j=0}^{\infty}$, supported on squares $E_j$, with the following properties: (a) $\|\cC_{\mu_j}\|_{\mu_j}\le1$; (b) $\|\mu_j\| \asymp\gamma(E_j)$; (c) $2E_j\cap2E_k=\emptyset$, $j\ne k$, \ $j,k\ge1$; (d) the measure $\mu=\sum_{j=0}^{\infty} \mu_j$ satisfies \eqref{mug}; (e) $\|\cC_{\mu}\|_{\mu}=\infty$.
\end{example}

\begin{proof} We use the idea of David-Semmes (see \cite[Example 8.7]{VE} for more detailed exposition). Let $N_0=0$, and let $\{N_{k}\}_{k=0}^{\infty}$ be a sequence of natural numbers such that $N_{k+1}-N_k\to\infty$ as $k\to\infty$. Start the construction with the unit square $E_0$ and make $N_1-N_0$ steps of the construction of the corner 1/4-Cantor set $E^{1/4}$. We get $4^{N_1-N_0}$ squares with side length $4^{-N_1}$. Choose one (any) of them, denote it by $Q_1$, and continue the construction only with this square. Other $4^{N_1-N_0}-1$ squares are the sets $E_j$ which already have been defined. For the chosen square $Q_1$ we make next $N_2-N_1$ steps of the construction of $E^{1/4}$, obtaining $4^{N_2-N_1}$ squares with side length $4^{-N_2}$. Again, continue the construction only for one of them, say, for a square $Q_2$, and so on.

Let $\mu_j$, $j=0,1,\dots$, be the 2-dimensional measure uniformly distributed on $E_j$ such that $\|\mu_{j}\|=c\ell_j$, where $\ell_j$ is the side length of $E_j$, and the absolute constant $c$ is chosen in such a way that $\|\cC_{\mu_j}\|_{\mu_j}=1$. Then properties (a), (b), (c) are obvious. To demonstrate (d) we notice that $E:=\bigcup_{j\ge 0} E_j$ is equal to $E_0$, and thus $\gamma(B\cap E) \asymp\diam(B\cap E)=:d_0$. On the other hand, for any $j\ge0$ and $d_j:=\diam(B\cap E_j)$, we have $\mu(B\cap E_j)\le c\ell_j^{-1}d_j^2<Cd_j$ (the density of $\mu_j$ is equal to $c/\ell_j$). Hence, $\mu(B\cap E)<C\sum_{j=0}^\infty d_j\asymp d_0$, and (d) is established.

Finally, to prove (e), we apply the operator $\mathcal{C}_\mu$ to the characteristic functions
$\chi_{Q_k}$, $k=0,1,\dots$ We have
$$
\|\cC_\mu(\chi_{Q_k})\|_{L^2(\mu)}=\|\cC_{\mu|Q_k}({\bf1})\|_{L^2(\mu)}\ge \|\cC_{\mu|Q_k}({\bf1})\|_{L^2(\mu|Q_k)}.
$$
But
$$
\|\cC_{\mu|Q_k}({\bf1})\|_{L^2(\mu|Q_k)}^2\ge c(N_{k+1}-N_k)4^{-N_k}
$$
with an absolute constant $c$ -- see \cite{MT}. Hence, $\|\cC_\mu\|_\mu\ge c(N_{k+1}-N_k)\to\infty$, and (e) is proved.
\end{proof}

Remark that the measures $\{\mu_{j}\}_{j=1}^{\infty}$ satisfy all assumptions of Theorem \ref{main1} except \eqref{mug}. Therefore, we have to add $\mu_0$ and change the structure of $\mu$.
\smallskip

Now we demonstrate that the condition $\|\mu_j\|\asymp \gamma(E_j)$ in Theorem \ref{main1} is essential.

\begin{example}\label{ex2} There exists a family of measures $\{\mu_{j}\}_{j=1}^{\infty}$ with the following properties: (a) $\|\cC_{\mu_j}\|_{\mu_j}\le1$; (b) $\|\mu_j\| \le C\gamma(E_j)$, where $E_j=\supp\mu_j$, and $E_j$ is either a square or a disc; (c) $2E_j\cap2E_k=\emptyset$, $j\ne k$; (d) the measure $\mu=\sum_{j=1}^{\infty} \mu_j$ satisfies \eqref{mug}; (e) $\|\cC_{\mu}\|_{\mu}=\infty$.
\end{example}

\begin{proof} We use the same construction as in Example \ref{ex1}, but with the following modifications. 1. The initial square $E_0$ now is not included into the collection $\{E_j\}$ of squares; thus, the squares are separated. 2. Besides the same squares $E_j$ and measures $\mu_j$, as in Example \ref{ex1}, we add additional discs $\widetilde E_j$ and measures $\widetilde \mu_j$ to satisfy \eqref{mug} (otherwise \eqref{mug} does not hold after exclusion $E_0$).

As before, set $N_0=0$, choose a sequence $\{N_{k}\}_{k=0}^{\infty}$ of natural numbers such that $N_{k+1}-N_k\to\infty$ as $k\to\infty$, and make $N_1-N_0$ steps of the construction of the corner 1/4-Cantor set $E^{1/4}$ starting with the unit square $E_0$. Let $E_{n,k}$, $k=0,\dots,4^n$, be the squares forming the $n$th generation in this construction (not all of them will be included into $\{E_j\}$). In each square $E_{n,k}$, $n=0,\dots,N_1$, $k=1,\dots,4^n$, place the disc $\widetilde D_{n,k}$ concentric with $E_{n,k}$, and with radius $\ell_n/10=4^{-n}/10$. On $\widetilde D_{n,k}$ we uniformly distribute a measure $\mu_{n,k}$ with $\|\mu_{n,k}\|=2^{-n}\cdot4^{-n}$. Then for $n$ bigger than a certain absolute $n_0$ we automatically have $\|\cC_{\mu_{n,k}}\|_{\mu_{n,k}}\le1$. For $n\le n_0$ we might need a small absolute positive  $c'$  such that $\|\cC_{c'\mu_{n,k}}\|_{c'\mu_{n,k}}\le1$.  We put then $\widetilde \mu_{n,k}:= c'\mu_{n,k}$ and achieve that  always $\|\cC_{\widetilde \mu_{n,k}}\|_{\widetilde\mu_{n,k}}\le1$. After $N_1-N_0$ steps choose one (any) of the squares $E_{N_1,k}$, and denote it by $Q_1$. Other $4^{N_1-N_0}-1$ squares $E_{N_1,k}$ and discs $\widetilde D_{n,k}$, $n=0,\dots,N_1$, $k=1,\dots,4^n$, are the sets $E_j$ which already have been defined. As before, $\mu_{N_1,k}$ is the 2-dimensional measure uniformly distributed on $E_{N_1,k}$ such that $\|\mu_{N_1,k}\|=c\ell_{N_1}$ and $\|\cC_{\mu_{N_1,k}}\|_{\mu_{N_1,k}}=1$.

For the chosen square $Q_1$ we continue the construction and make the next $N_2-N_1$ steps of the construction of $E^{1/4}$, obtaining the squares $E_{N_2,k}$, $k=1,\dots,4^{N_2-N_1}$, with side length $\ell_{N_2}:=4^{-N_2}$ and with measures $\mu_{N_2,k}$ such that $\|\mu_{N_2,k}\|=c\ell_{N_2}$. Besides these squares, we get discs $\widetilde D_{n,k}$, $n=N_1,\dots,N_2$, $k=1,\dots,4^{n-N_1}$ of radii $4^{-n}/10$, concentric with $E_{n,k}$ and supporting the measures $\widetilde \mu_{n,k}$, $\|\widetilde \mu_{n,k}\|=c'2^{-n}\cdot4^{-n}$. Again, continue the construction only for one of them, and so on.

We have to prove only (d). Fix a disc $B$. Suppose that $\widetilde D_{0,1}\subset B$. Since $\mu(B)\le\|\mu\|\le C$ and $\g(B\cap E)\ge \g(\widetilde D_{0,1})=1/10$, \eqref{mug} holds. Suppose now that $\widetilde D_{0,1}\not\subset B$, and that $B$ contains at least one disc $\widetilde D_{n,k}$. Choose a maximal disc in $B$, say, $\widetilde D_{n_1,k_1}$. Then ``the parent'' $E_{n_1-1,k_1'}$ of the square $E_{n_1,k_1}$ (that is the square of the previous generation containing $E_{n_1,k_1}$) does not lie in $B$: otherwise $\widetilde D_{n_1,k_1}$ would not be a maximal disc in $B$. Set $G_1:=E_{n_1-1,k_1'}\cap B$. Now choose a maximal disc $\widetilde D_{n_2,k_2}\subset (B\setminus G_1)$ (if such a disc exist). Its ``parent'' $E_{n_2-1,k_2'}$ and the set $G_1$ are disjoint. Hence, $E_{n_2-1,k_2'}\not\subset B$. Set $G_2:=E_{n_2-1,k_2'}\cap B$. Continuing in such a way, we obtain a sequence of sets $G_j$ with the following properties. (i) $2G_i\cap2G_j=\emptyset$, $i\ne j$, and one may place them into $\lambda$-separated discs, $\la>1$. (ii) For each $G_j$,
$$
\mu(G_j)\le C\ell_{n_j}=C4^{-n_j}\asymp\g(\widetilde D_{n_j,k_j})\le\g(G_j).
$$
(iii) All squares $E_{N_i,k}$ and discs $\widetilde D_{n,k}$ contained in $B$ are contained in $\cup_j G_j$.

Also, it might be some discs $\widetilde D_{n,k}$ and squares $E_{N_i,k}$ intersecting $B$. Each of them form a separate set $G_j:=\widetilde D_{n,k}\cap B$ or $E_{N_i,k}\cap B$. These sets are $\la$-separated as well, and
$$
\mu(G_j)\le C\diam(G_j)\le C'\g(G_j)
$$
(see the proof of Example \ref{ex1}). Moreover, all sets $G_j$ satisfy the conditions of Corollary \ref{circle1} which yields the estimate
$$
\mu(B)\le\sum_j\mu(G_j)\le C\sum_j\g(G_j)\le C'\g(B\cap E),
$$
and the proof is completed.
\end{proof}

Now we demonstrate the sharpness of Corollary \ref{geom}.
\begin{prop}\label{prop53}
Without any of assumptions (a)--(c) Corollary \ref{geom} is incorrect.
\end{prop}
\begin{proof} {\bf1.} Suppose that the condition (a) in Corollary \ref{geom} is missed. Then we may use the same sets $\{E_{j}\}_{j=1}^{\infty}$ as in Example \ref{ex1}, only without the initial square $E_0$. As $D_j$ we take discs containing $E_j$ and slightly bigger than $E_j$. A proof of (b) and (c) is not difficult, and we leave it to the reader. At the same time,
\begin{equation}\label{f52}
(\text{length of any curve connecting all discs in }Q_k)/4^{-N_k}\to\infty\ \text{ as }k\to\infty
\end{equation}
(the squares $Q_k$ are defined in the proof of Example \ref{ex1}).

{\bf2.} Consider the case when (b) is missed. Let $D_j$ be the discs $\widetilde D_{n,k}$ in Example \ref{ex2}, enumerated in the order of non-increase of their radii $r_j$. In each $D_j$ we place the disc $E_j$ concentric with $D_j$ and of radius $\frac1{20}4^{-j}$. To prove (a), fix a disc $B$. For discs $E_j$ intersecting $\partial B$, \eqref{almadd} holds by Corollary \ref{circle1} (with $B\cap E_j$ as $E_j$). Let $j_0:=\{\min j: E_j\subset B\}$. Then
\begin{align*}
\sum_{j:B_j\subset B}\g(B\cap E_j)&\le\sum_{j=j_0}^\infty\g(E_j)=\sum_{j=j_0}^\infty\frac1{20}4^{-j}\\
&=\frac1{15}4^{-j_0}=\frac83\g(E_{j_0})\le\frac83\g(B\cap E),
\end{align*}
and (a) is proved. The proof of (c) is easy and we skip it. At the same time, \eqref{f52} holds for discs $D_j$ in $Q_k$.

{\bf3.} Finally, suppose that (c) is missed. The counterexample in this case is not based on Cantor-type constructions. Given $\ell>0$, $N\in\mathbb N$, $N\ge4$, consider the square $Q_{\ell}=[0,\ell]\times[0,\ell]$ and points
$$
x_i=\frac{\ell}{N-1}i,\quad y_j=\frac{\ell}{N-1}j,\quad i,j=0,1,\dots,N-1.
$$
Let $E_{ij}$ be discs centered at the points $(x_i,y_j)$ with radii $\ell/N^2$, and let $\cE:=\bigcup_{i,j}E_{ij}$, $D_{ij}:=2E_{ij}$. Fix a disc $B$. Set $\cE_B=\{\bigcup E_{ij}:E_{ij}\subset B\}$ ($\cE_B$ might be empty). If $\cE_B\ne\emptyset$, we have
$$
\bigg|\int\frac{d(\HH^1|\partial\cE_B)(\xi)}{\xi-z}\bigg|<C,\quad z\in\C.
$$
Hence,
$$
\g(\cE\cap B)\ge\g(\cE_B)\ge c\HH^1(\partial\cE_B)=c'\sum_{E_{ij}\subset B}\g(E_{ij})\ge c''\sum\g(E_{ij}\cap B).
$$
If $\cE_B=\emptyset$, \eqref{almadd} holds as well (for example, by Corollary \ref{circle1}).
At the same time, the length of any curve, intersecting all discs $D_{ij}$, is at least $C\ell N$.

Now we construct a series of squares $Q_{\ell_k}$ with $\ell_k=\frac1{10}2^{-k}$, $N_k=k\,2^k$, centered at points $1/k^2$, and the corresponding sets $\cE^{(k)}$ and discs $D_{ij}^{(k)}$, $E_{ij}^{(k)}$. One may place $\cE^{(k)}$ into $\la$-separated discs centered at points $1/k^2$. Set $E:=\bigcup_k\cE^{(k)}$ (that is $E$ is the union of all discs $E_{ij}^{(k)}$). By Theorem \ref{sup},
$$
\g(B\cap E)\ge c\sum_k\g(B\cap\cE^{(k)}).
$$
Thus, to prove \eqref{almadd}, we have to establish almost additivity of $\g$ for each $\cE^{(k)}$ separately, that was done above. The validity of (b) is obvious. But all discs $D_{ij}^{(k)}$ in $E$ cannot be connected by an AD regular curve.
\end{proof}

\section{Proof of Theorem \ref{superth}}
\label{supth}

It is known that a  compact chord-arc curve is a bi-lipschitz image of a straight segment, see \cite{Po}, Chapter 7. On the other hand  analytic capacity can be only finitely distorted by bi-lipschitz maps. This is a difficult result by X. Tolsa, \cite{To3}. So if we allow the separation constant $\la>1$ to depend on the Lipschitz constant of our chord-arc curve (so, the separation of the discs to be large if the constant of the curve is large), then we can obtain Theorem \ref{superth} directly from Theorem \ref{sup}.
However, we want to avoid the dependence of the separation constant on the chord-arc constant. Then we need another proof, which follows.

The Melnikov--Menger curvature of a positive Borel measure $\mu$ in $\C$ is defined as
$$
c^2(\mu)=\iiint\frac1{R^2(x,y,z)}\,d\mu(x)\,d\mu(y)\,d\mu(z),
$$
where $R(x,y,z)$ is the radius of the circle passing through points $x,y,z\in\C$, with $R(x,y,z)=\infty$ if $x,y,z$ lie on the same straight line (in particular, if two of these points coincide). This notion was introduced by Melnikov \cite{M}. The following relation characterizes the analytic capacity in terms of the curvature of a measure  \cite{Tolsa-sem}, \cite[p.~104]{To2}, \cite{Volberg}, \cite{Tolsa-book} : for any compact set $F$ in $\C$,
\begin{equation}\label{f61}
\g(F)\asymp\sup\{\mu(F):\ \supp\mu\subset F,\ \mu\in\Sigma,\ c^2(\mu)\le \mu(F)\},
\end{equation}
where $\Sigma$ is the class of measures of linear growth defined in \eqref{gop}.

\begin{lemma}[Main Lemma]\label{le61}
Let $D_j=D(x_j,r_j)$ be discs with centers on a chord-arc curve $\Gamma$, such that $\la D_j \cap \la D_k = \emptyset$, $j\not=k$, for some $\la>1$. Let $\mu_j$ be positive measures with the following properties: (1) $\supp\mu_j\subset D_j$; (2) $\mu_j(B_j)=:\|\mu_j\|\le r_j$. Then for $\mu=\sum\mu_j$ we have
\begin{equation}\label{f62}
c^2(\mu)\le \sum_j c^2(\mu_j)+C\|\mu\|,\quad C=C(\la,A_0),
\end{equation}
where $A_0$ is the constant of $\Gamma$.
\end{lemma}

At the beginning let us show that Theorem \ref{superth} is a direct consequence of Main Lemma and \eqref{f61}.
\begin{proof}[Proof of Theorem \ref{superth}]
Consider measures $\mu_j$ participating in \eqref{f61} for $F=E_j$, $j=1,\dots$ Then $\mu(D(x,r))\le Cr$ for any disc $D$, where $C=C(A_0)$ and $\mu=\sum\mu_j$. To prove this assertion, we fix a disc $D=D(x,r)$ and divide all discs $D_j$ onto two groups: $\cD_1:=\{D_j:D_j\cap D\ne\emptyset,\ r_j\le r\}$, $\cD_2:=\{D_j:D_j\cap D\ne\emptyset,\ r_j>r\}$. Since $\Gamma$ is chord-arc, $\sum_{D_j\in\cD_1}r_i\le Cr$, $C=C(A_0)$. It is easy to see that $\#\cD_2\le6$. Hence,
$$
\mu(D)\le\sum_{D_j\in\cD_1}\mu(D_j)+\sum_{D_j\in\cD_2}\mu(D_j\cap D)
\le \sum_{D_j\in\cD_1}r_j+6\mu(D)<Cr.
$$
Furthermore, Main Lemma implies the inequality $c^2(\mu)\le C\|\mu\|$, $C=C(\la,A_0)$. Thus, the measure $c\mu$ with an appropriate constant $c$ depending on $\la,A_0$, participates in \eqref{f61} for $F=E=\cup E_j$. So, $\g(E)\ge c\|\mu\|$, that implies Theorem \ref{superth}.
\end{proof}

\begin{proof}[Proof of Lemma \ref{le61}]
It is enough to consider the case of a finite set of discs $B_j$, $j=1,\dots,N$. We assume that these discs are enumerated in the order of increase of the natural parameters of their centers.

Let $\Gamma_j$ be arcs of $\Gamma$ such that $\Gamma_j\subset D_j$ and $\HH^1(\Gamma_j)=\mu(D_j)$. Let $\sigma_j:=\HH^1|\Gamma_j$ and $\sigma:=\sum\sigma_j$, so that $\sigma(D_j)=\mu(D_j)$. Obviously,
$$
c^2(\mu)=\bigg(\sum_j\iiint_{D_j^3}+\iiint_{\C^3\setminus\bigcup_j D_j^3}\bigg)\frac1{R^2(x,y,z)}\,d\mu(x)\,d\mu(y)\,d\mu(z)=:I_1+ I_2.
$$
Since $I_1=\sum_j c^2(\mu_j)$, we have to estimate only $I_2$. Our proof is based on the comparison of $I_2$ and the corresponding integral with respect to $\sigma$:
$$
\bar I_2:=\iiint_{\C^3\setminus\bigcup_j D_j^3}\frac1{R^2(x,y,z)}\,d\sigma(x)\,d\sigma(y)\,d\sigma(z).
$$
Notice that
\begin{equation}\label{f63}
\bar I_2<c^2(\sigma)\le C\|\sigma\|,\quad C=C(A_0).
\end{equation}
The last inequality is a consequence of two well-known facts. (a) The boundedness of the Cauchy operator $\cC_{\HH^1|\Gamma}$ on chord-arc curves -- see \cite[p.~330]{MV}. In particular,
$$
\|\cC_\s^\e{\bf1}\|_{L^2(\s)}^2\le\|\cC_{\HH^1|\Gamma}\chi_{\cup\Gamma_j}\|_{L^2(\HH^1|\Gamma)}^2 \le C\|\chi_{\cup\Gamma_j}\|_{L^2(\HH^1|\Gamma)}^2=C\|\s\|,\quad\e>0,
$$
where $C$ depends only on $A_0$. (b) The connection between the curvature of a measure and the norm of a Cauchy potential:
$$
\|\cC_\mu^\e{\bf1}\|_{L^2(\mu)}^2=\frac16c_\e^2(\mu)+O(\|\mu\|)
$$
for any measure $\mu\in\Sigma$ uniformly in $\e$ -- see for example \cite{To2}. Here $c_\e^2(\mu)$ is the truncated version of $c^2(\mu)$ defined in the same way as $c_\e^2(\mu)$, but the triple integral is taken over the set $\{(x,y,z)\in\C^3:|x-y|,|y-z|,|x-z|>\e\}$. This equality with $\mu=c\,\s\in\Sigma$, and the previous relations imply \eqref{f63}.

Obviously,
$$
I_2=\bigg(\iiint_{\Omega_1}+\iiint_{\Omega_2}\bigg)\frac1{R^2(x,y,z)}\,d\mu(x)\,d\mu(y)\,d\mu(z)=:I_{2,1}+ I_{2,2}\,,
$$
where
\begin{align*}
\Omega_1&:=\{D_j\times D_k\times D_l: j=k\ne l\vee j\ne k=l\vee j=l\ne k\},\\
\Omega_2&:=\{D_j\times D_k\times D_l: j\ne k,\ k\ne l,\ j\ne l\}.
\end{align*}
To estimate the integral over $\Omega_1$, it's sufficient to consider the subset
$$
\Omega'_1:=\{D_j\times D_k\times D_l: j\ne k=l\}.
$$
For $x\in D_j=D(x_j,r_j),\ y,z\in D_k,\ j\ne k$, we have
$$
2R(x,y,z)\ge|x-y|\ge c(r_j+r_{j+1}+\dots+r_k),\quad c=c(\la,A_0)
$$
(here we assume that $j<k$; the case $k<j$ is analogous). Then
\begin{multline*}
\iiint_{\Omega'_1}\frac1{R^2(x,y,z)}\,d\mu(x)\,d\mu(y)\,d\mu(z)\le C\bigg[\sum_{j=1}^{N-1}\|\mu_j\|\sum_{k=j+1}^{N}\frac{\|\mu_k\|^2}{(r_j+r_{j+1}+\dots+r_k)^2}\\
+\sum_{j=1}^{N-1}\|\mu_{N+1-j}\|\sum_{k=j+1}^{N}\frac{\|\mu_{N+1-k}\|^2} {(r_{N+1-j}+r_{j+1}+\dots+r_{N+1-k})^2}=:C[S_{N,1}+S_{N,2}].
\end{multline*}
Estimates for both terms on the right are the same. We estimate $S_{N,1}$ using the induction with respect to $N$.

1. $N=2$. Then
$$
S_{N,1}=\|\mu_1\|\cdot\frac{\|\mu_2\|^2}{(r_1+r_2)^2}\le\|\mu_1\|\le\|\mu_1\|+\|\mu_2\|.
$$

2. Suppose that the inequality
\begin{equation}\label{f64}
S_{N,1}=\sum_{j=1}^{N-1}\|\mu_j\|\sum_{k=j+1}^{N}\frac{\|\mu_k\|^2}{(r_j+\dots+r_k)^2}\le\|\mu_1\|+\dots+\|\mu_N\|
\end{equation}
holds for some $N\ge2$. For $N+1$ discs we have
\begin{align*}
S_{N+1,1}&=S_{N,1}+\sum_{j=1}^{N}\|\mu_j\|\frac{\|\mu_{N+1}\|^2}{(r_j+\dots+r_{N+1})^2}\\
&\le S_{N,1}+\|\mu_{N+1}\|^2\sum_{j=1}^{N}\frac{r_j}{(r_j+\dots+r_{N+1})^2}.
\end{align*}
The last sum is dominated by the integral
$$
\int_0^\infty\frac{dt}{(r_{N+1}+t)^2}=\frac1{r_{N+1}}.
$$
Hence,
$$
S_{N+1,1}\le S_{N,1}+\|\mu_{N+1}\|^2/r_{N+1}\le\|\mu_1\|+\dots+\|\mu_{N+1}\|.
$$
Thus, we proved \eqref{f64} and therefore estimated the triple integral over $\Omega_1$.

By symmetry,
$$
\iiint_{\Omega_2}\frac1{R^2(x,y,z)}\,d\mu(x)\,d\mu(y)\,d\mu(z) = 6\iiint_{\Omega'_2}\frac1{R^2(x,y,z)}\,d\mu(x)\,d\mu(y)\,d\mu(z),
$$
where $\Omega'_2:=\{D_j\times D_k\times D_l: j< k<l\}$. Moreover, we may restrict ourself by the integration over
$$
\Omega'_{2,1}:=\{D_j\times D_k\times D_l: j< k<l,\ r_j+\dots+r_k\ge\tfrac12(r_j+\dots+r_l)\}.
$$
Indeed, if we prove the inequality
\begin{equation}\label{f65}
\iiint_{\Omega'_{2,1}}\frac1{R^2(x,y,z)}\,d\mu(x)\,d\mu(y)\,d\mu(z)\le C\|\mu\|
\end{equation}
with $C=C(\la,A_0)$, then using the inverse parametrization of $\Gamma$, we get the same estimate for the triple integral over
$$
\Omega'_{2,2}:=\{D_j\times D_k\times D_l: j< k<l,\ r_k+\dots+r_l\ge\tfrac12(r_j+\dots+r_l)\}
$$
(here we use the same numeration of discs as before). Since $\iiint_{\Omega'_2}\le\iiint_{\Omega'_{2,1}}+\iiint_{\Omega'_{2,2}}$, \eqref{f64} and \eqref{f65} imply \eqref{f62}.

Fix indices $j,k,l$. For any triples $(x,y,z),(x',y',z')\in D_j\times D_k\times D_l$, the sine of the angle between the intervals $(y,z)$ and $(y',z')$ does not exceed
$$
C\,\frac{r_k+r_l}{r_k+\dots+r_l},\quad C=C(\la,A_0).
$$
For the angle between the intervals $(x,z)$ and $(x',z')$ we have $C\,\frac{r_j+r_l}{r_j+\dots+r_l}$. Denote by $\al$, $\al'$ the angles at $z$, $z'$ of the triangles $x,y,z$ and $x',y',z'$, correspondingly. Since $\sin(\al+\beta+\g)\le\sin\al+\sin\beta+\sin\g$ as $\al,\beta,\g\in[0,\pi]$, we get the estimate
$$
\sin\al<\sin\al'+C\,\frac{r_k+r_l}{r_k+\dots+r_l}+C\,\frac{r_j+r_l}{r_j+\dots+r_l}.
$$
Hence,
$$
\frac1{R(x,y,z)}=\frac{2\sin\al}{|x-y|}<\frac{C}{|x'-y'|}\bigg[2\sin\al'+ \frac{r_k+r_l}{r_k+\dots+r_l}+\frac{r_j+r_l}{r_j+\dots+r_l}\bigg].
$$
Therefore,
\begin{multline*}
\iiint_{\Omega_2}\frac1{R^2(x,y,z)}\,d\mu(x)\,d\mu(y)\,d\mu(z) \le C\bigg[\iiint_{\Omega'_2}\frac1{R^2(x',y',z')}\,d\s(x')\,d\s(y')\,d\s(z')\\
+\sum_{j=1}^{N-2}\|\mu_j\|\sum_{l=j+2}^{N}\sum_{k=j+1}^{l-1} \frac{r_k^2\,\|\mu_{k}\|\,\|\mu_{l}\|}{(r_j+\dots+r_{k})^2(r_k+\dots+r_{l})^2}\\
+\sum_{j=1}^{N-2}\|\mu_j\|\sum_{l=j+2}^{N}\sum_{k=j+1}^{l-1} \frac{r_l^2\,\|\mu_{k}\|\,\|\mu_{l}\|}{(r_j+\dots+r_{k})^2(r_k+\dots+r_{l})^2}\\
+\sum_{j=1}^{N-2}\|\mu_j\|r_j^2\sum_{l=j+2}^{N}\sum_{k=j+1}^{l-1} \frac{\|\mu_{k}\|\,\|\mu_{l}\|}{(r_j+\dots+r_{k})^2(r_j+\dots+r_{l})^2}\bigg]
=:C[I+S^{(1)}+S^{(2)}+S^{(3)}].
\end{multline*}
By \eqref{f63}, $I\le c^2(\s)\le C\|\s\|$. We estimate each of sums separately. Write $S^{(1)}$ as
$$
S^{(1)}=\sum_{j=1}^{N-2}\|\mu_j\|\sum_{k=j+1}^{N-1}\sum_{l=k+1}^{N} \frac{r_k^2\,\|\mu_{k}\|\,\|\mu_{l}\|}{(r_j+\dots+r_{k})^2(r_k+\dots+r_{l})^2}\,.
$$
Since the inner sum over $l$ does not exceed
$$
\frac{r_k^2\,\|\mu_{k}\|}{(r_j+\dots+r_{k})^2}\int_{r_k}^\infty\frac{dx}{x^2}=
\frac{r_k\,\|\mu_{k}\|}{(r_j+\dots+r_{k})^2},
$$
we get the estimate
\begin{equation}\label{f66}
\begin{split}
S^{(1)}&\le\sum_{j=1}^{N-2}\|\mu_j\|\sum_{k=j+1}^{N-1} \frac{r_k\,\|\mu_{k}\|}{(r_j+\dots+r_{k})^2} \le\sum_{k=2}^{N-1}\sum_{j=1}^{k-1} \frac{r_j\, r_k\,\|\mu_{k}\|}{(r_j+\dots+r_{k})^2}\\
&\le\sum_{k=2}^{N-1}\|\mu_{k}\|\sum_{j=1}^{k-1}r_k\int_{r_k}^\infty\frac{dx}{x^2}=
\sum_{k=2}^{N-1}\|\mu_{k}\|<\|\mu\| \,.
\end{split}
\end{equation}
Now we will use the possibility to consider only those $k$ for which $r_j+\dots+r_k\ge\frac12(r_j+\dots+r_l)$ (the set of such $k$ can be empty). Suppose that the last inequality holds for $p\le k\le l-1$. Then
\begin{multline*}
\sum_{j=1}^{N-2}\|\mu_j\|\sum_{l=j+2}^{N}\sum_{k=p}^{l-1} \frac{r_l^2\,\|\mu_{k}\|\,\|\mu_{l}\|}{(r_j+\dots+r_{k})^2(r_k+\dots+r_{l})^2}\\
\le4\sum_{j=1}^{N-2}\|\mu_j\|\sum_{l=j+2}^{N}\sum_{k=p}^{l-1} \frac{r_l^2\,\|\mu_{k}\|\,\|\mu_{l}\|}{(r_j+\dots+r_l)^2(r_k+\dots+r_{l})^2}\\
\le4\sum_{j=1}^{N-2}\|\mu_j\|\sum_{l=j+2}^{N}\frac{r_l\,\|\mu_l\|}{(r_j+\dots+r_l)^2}
\end{multline*}
(we estimate the sum with respect to $k$ in the same way as above). We may deal with the last double sum as in \eqref{f66}, or notice that this sum does not exceed
$$
\sum_{j=1}^{N-2}\|\mu_j\|\bigg[1+\sum_{l=j+1}^{N-1}\frac{r_l\,\|\mu_l\|}{(r_j+\dots+r_l)^2}\bigg],
$$
and use \eqref{f66} directly. Finally,
$$
S^{(3)}\le\sum_{j=1}^{N-2}\|\mu_j\|\sum_{l=j+2}^{N} \frac{r_j^2\,\|\mu_{l}\|}{(r_j+\dots+r_{l})^2r_j}<\sum_{j=1}^{N-2}\|\mu_j\|<\|\mu\|.
$$
Lemma \ref{le61} is proved.
\end{proof}

\section{Question on super-additivity}\label{q}

We make more accurate the question posed in Section \ref{int}. In Theorems \ref{superth}, \ref{sup} discs were $\lambda$-separated, where $\lambda>1$. But what if they are just disjoint? Namely, let
$D_j$ be circles with centers on a chord-arc curve with constant $A_0$ (or even on the real line $\R$), such that $ D_j \cap  D_k = \emptyset$, $j\not=k$. Let $E_j\subset D_j$ be arbitrary compact sets. Is it true that there exists a constant  $c=c(\la,A_0)>0$, such that
$$
\gamma\big(\bigcup E_j\big) \ge c \sli_j \gamma(E_j)\,?
$$
We cannot either prove or construct a counter-example to this claim.

\end{document}